\documentclass[a4paper,12pt,reqno]{amsart}

\usepackage[utf8]{inputenc}
\usepackage[english,latin]{babel}
\usepackage{amsfonts}
\usepackage{amsmath}
\usepackage{amssymb}
\usepackage{amsthm}
\usepackage{enumerate}
\usepackage{geometry}

\newcommand*{\1}{1\!\!\,\mathrm{I}}

\newcommand*{\abs}[1]{\left|#1\right|}
\newcommand*{\norm}[1]{\left\|#1\right\|}

\newcommand*{\Prob}[1]{\mathbf{P} \left\lbrace #1\right\rbrace}
\newcommand*{\E}{\mathbf{E}}

\newcommand*{\qv}[1]{\left\langle #1\right\rangle}
\newcommand*{\jqv}[2]{\left\langle #1,#2\right\rangle}

\newcommand*{\ve}{\varepsilon}
\newcommand*{\vf}{\varphi}
\newcommand*{\mbN}{\mathbb{N}}

\newcommand*{\mbR}{\mathbb{R}}

\newcommand*{\mcB}{\mathcal{B}}
\newcommand*{\mcF}{\mathcal{F}}

\newcommand*{\mcN}{\mathcal{N}}

\theoremstyle{plain}
\newtheorem{theorem}{Theorem}[section]
\newtheorem{lemma}[theorem]{Lemma}

\newtheorem{corollary}[theorem]{Corollary}

\theoremstyle{definition}
\newtheorem{definition}[theorem]{Definition}

\theoremstyle{remark}
\newtheorem{remark}[theorem]{Remark}

\title{The level-crossing intensity for the density of the image of the Lebesgue measure under the action of a Brownian stochastic flow}
\author{V.~V.~Fomichov}
\address{Vladimir~Fomichov: Institute of Mathematics, National Academy of Sciences of Ukraine, Tereshchenkivska str.~3, Kiev~01004, Ukraine}
\email{v-vfom@imath.kiev.ua}
\keywords{Brownian stochastic flows, level-crossing intensity, stationary processes}

\begin{document}
\selectlanguage{english}

\begin{abstract}
In this paper we compute the level-crossing intensity for the density of the image of the Lebesgue measure under the action of a Brownian stochastic flow which is a smooth approximation of the Arratia flow and determine its asymptotic behaviour as the height of the level tends to infinity.
\end{abstract}

\maketitle

\section{Introduction}
\label{section1}

Let us consider the following stochastic integral equation:
\begin{equation}
\label{equation1}
x(u,t)=u+\int\limits_0^t \int\limits_\mbR \vf(x(u,s)-q)\,W(dq,ds),\quad t\geqslant 0,
\end{equation}
where $u\in\mbR$ is a fixed parameter, $W$ is a Wiener sheet on $\mbR \times \mbR_+$, and the function $\vf\colon\mbR \rightarrow \mbR_+$ satisfies the conditions:
\begin{enumerate}
\item[(i)]
$\vf\in C_K^\infty(\mbR,\mbR_+)$, i.~e. $\vf$ is non-negative, infinitely differentiable and has compact support;

\item[(ii)]
$\vf(q)=\vf(-q)$, $q\in\mbR$;

\item[(iii)]
$\int\limits_\mbR \vf^2(q)\,dq=1$.
\end{enumerate}

Under such conditions on the function $\vf$ equation~\eqref{equation1} has a unique strong solution for every $u\in\mbR$. Moreover, for some set $\Omega$ of full probability (without loss of generality we will assume that it is the whole set of outcomes) the mappings
\begin{equation}
\label{equation2}
x(\omega,\cdot,t)\colon\mbR \rightarrow \mbR,\quad t\geqslant 0,\quad \omega\in\Omega,
\end{equation}
are $C^\infty$-diffeomorphisms and the stochastic flow $\{\vf_{s,t}(u),\; u\in\mbR,\; 0\leqslant s\leqslant t<+\infty\}$ given by
\begin{equation}
\label{equation3}
\vf_{s,t}(\omega,u):=x(\omega,x^{-1}(\omega,u,s),t),\quad u\in\mbR,\quad 0\leqslant s\leqslant t<+\infty,\quad \omega\in\Omega,
\end{equation}
where $x^{-1}(\omega,\cdot,s)$ is the inverse of the mapping $x(\omega,\cdot,s)$ (as a rule, we will omit the variable $\omega$), is a Brownian stochastic flow of $C^\infty$-diffeomorphisms (see~\cite{Dorogovtsev2004}, \cite{Dorogovtsev2007}, and also~\cite{Kunita}).

The covariance function of this Brownian stochastic flow is given by
$$
\Phi(z):=\int\limits_\mbR \vf(z+q)\vf(q)\,dq,\quad z\in\mbR.
$$
In other words, for any $u,v\in\mbR$ the joint quadratic variation of the Wiener processes $\{x(u,t),\; t\geqslant 0\}$ and $\{x(v,t),\; t\geqslant 0\}$ has the form
$$
\jqv{x(u,\cdot)}{x(v,\cdot)}_t=\int\limits_0^t \Phi(x(u,s)-x(v,s))\,ds,\quad t\geqslant 0.
$$

Note that the function $\Phi$ takes value $1$ at point $0$ and has compact support, whose diameter is not greater than $2d(\vf)$, where $d(\vf)$ is the diameter of the support of the function $\vf$. Therefore, as $d(\vf)$ tends to zero, the function $\Phi$ converges pointwise to the function
$$
\1_{\{0\}}(z)=
\begin{cases}
1,\quad \text{if $z=0$,}\\
0,\quad \text{if $z\neq 0$.}
\end{cases}
$$

In~\cite{Dorogovtsev2004} it was shown that for any $n\in\mbN$ and for any $u_1,\ldots,u_n\in\mbR$ in the space $C([0;1],\mbR^n)$ the following weak convergence takes place:
$$
(x(u_1,\cdot),\ldots,x(u_n,\cdot))\stackrel{w}{\longrightarrow} (x_0(u_1,\cdot),\ldots,x_0(u_n,\cdot)),\quad d(\vf)\to 0+,
$$
where $\{x_0(u,t),\; u\in\mbR,\; t\geqslant 0\}$ is the Arratia flow, i.~e. a Brownian stochastic flow with covariance function $\1_{\{0\}}$.

Recall that the Arratia flow was constructed in~\cite{Arratia} as the weak limit of families of coalescing simple random walks and that informally it can be described as a flow of Brownian particles in which any two particles move independently until they meet and after that coalesce and move together. Moreover, it is known (see~\cite{Arratia}, \cite{Harris}) that for any $t>0$ and for any interval $I\subset\mbR$ the set $x_0(\mbR,t) \cap I$ is finite almost surely.

Consider the random measures
$$
\lambda_t:=\lambda \circ x^{-1}(\cdot,t),\quad t\geqslant 0,
$$
where $\lambda$ is the one-dimensional Lebesgue measure. Due to the diffeomorphic property of the mappings~\eqref{equation2}, these measures are absolutely continuous. Furthermore, it is easy to check that for the corresponding densities the following representation takes place (strict positivity of the derivative $\frac{\partial x}{\partial u}(u,t)$ is implied by Lemma~\ref{lemma3.1} stated below)
$$
\dfrac{d\lambda_t}{d\lambda}(u)\equiv p_t(u)=\dfrac{1}{\dfrac{\partial x}{\partial u}(x^{-1}(u,t),t)},\quad u\in\mbR,\quad t\geqslant 0.
$$

Now if we set
$$
\lambda_t^0:=\lambda \circ x_0^{-1}(\cdot,t),\quad t\geqslant 0,
$$
it is natural to expect that, as $d(\vf)\to 0+$, the random measures $\lambda_t$ converge in some sense to the random measures $\lambda_t^0$. Therefore, the domains of concentration of the measures $\lambda_t$ or, in other words, the domains in which the densities $p_t$ take great values correspond to the atoms of the measures $\lambda_t^0$ (which obviously coincide with the clusters of the Arratia flow).

It is reasonable to start the study of such high level exceedances of the densities $p_t$ by investigating their level-crossing intensities. Let us recall the corresponding definition.

Let $\{\xi(u),\; u\in\mbR\}$ be a stochastic process, whose trajectories with probability one are not identically equal to a given number $c\in\mbR$ on any interval. Then the number $N([0;1];c)$ of crossings by the random process $\xi$ of the level $c$ on the interval $[0;1]$ is well-defined and the number
$$
\mu(c):=\E N([0;1];c)
$$
is called the intensity of crossings of the level $c$ by the stochastic process $\xi$.

For some classes of stochastic processes $\{\xi(u),\; u\in\mbR\}$ the value of $\mu(c)$ can be computed with the help of the well-known Rice formula (see~\cite{LeadbetterSpaniolo}, where upcrossings are considered, and the references therein). As an example we formulate the following theorem, noting on the way that in this theorem and further $\pi[\eta](\cdot)$ and $\pi[\eta,\zeta](\cdot,\cdot)$ stand for the densities of the distribution of a random variable $\eta$ and of the joint distribution of random variables $\eta$ and $\zeta$ respectively.

\begin{theorem}
\cite{AzaisWschebor}
Let a stochastic process $\{\xi(u),\; u\in\mbR\}$ have continuously differentiable paths and satisfy the following three conditions:
\begin{enumerate}
\item[\textup{(i)}]
the mapping $(u,z)\mapsto\pi[\xi(u)](z)$ is continuous in $u\in [0;1]$ and $z\in U_c$, where $U_c$ is some neighbourhood of $c\in\mbR$;

\item[\textup{(ii)}]
the mapping $(u,z_1,z_2)\mapsto\pi[\xi(u),\xi'(u)](z_1,z_2)$ is continuous in $u\in [0;1]$, $z_1\in U_c$ and $z_2\in\mbR$;

\item[\textup{(iii)}]
$\E w(\xi';\delta)\to 0$ as $\delta\to 0+$, where $w(\xi';\delta)$ is the modulus of continuity of the stochastic process $\xi'$ on the interval $[0;1]$.
\end{enumerate}
Then for any $c\in\mbR$
\begin{equation}
\label{equation4}
\mu(c)=\int\limits_0^1 \int\limits_0^{+\infty} \abs{z} \cdot \pi[\xi(u),\xi'(u)](c,z)\,dz\,du.
\end{equation}
\end{theorem}

However, the conditions on the stochastic process ensuring the validity of Formula~\eqref{equation4} for \emph{every} level $c\in\mbR$ are often difficult to verify (for instance, in the above theorem such is condition~(iii)), and they were checked only for stochastic processes similar to Gaussian. It is easier to verify the conditions ensuring the validity of similar formulae for \emph{almost every} level $c\in\mbR$.

In this paper we use the following version of Rice's formula (see~\cite[Exercise~3.8]{AzaisWschebor}; cf.~\cite[Formula~(3)]{LeadbetterSpaniolo}) also called Banach's formula.

\begin{theorem}
\label{theorem1.2}
\cite{AzaisWschebor}
Let a stochastic process $\{\xi(u),\; 0\leqslant u\leqslant 1\}$ be such that almost all its trajectories are absolutely continuous, for any $u\in [0;1]$ the distribution of the random variable $\xi(u)$ has a density $\pi[\xi(u)](\cdot)$, and the conditional expectation $\E(\abs{\xi'(u)} \mid \xi(u)=c)$ is well-defined. Then for almost every $c\in\mbR$
$$
\mu(c)=\int\limits_0^1 \E(\abs{\xi'(u)} \mid \xi(u)=c) \cdot \pi[\xi(u)](c)\,du.
$$
\end{theorem}

The main aim of this paper is to verify that the level-crossing intensity $\mu_t(c)$ (with $c>0$) of the stochastic process $\{p_t(u),\; u\in\mbR\}$ is well-defined for any $t>0$, and to prove that
\begin{equation}
\label{equation5}
\mu_t(c)=\overline{\mu}_t(c)\qquad \text{for a.~e. $c>0$,}
\end{equation}
where
$$
\overline{\mu}_t(c)=\dfrac{\sqrt{2L''} \cdot e^{\frac{\pi^2}{2L't}-\frac{L't}{8}}}{\pi L'\sqrt{\pi t}} \cdot \dfrac{1}{\sqrt{c}} \cdot \int\limits_0^{+\infty} \dfrac{e^{-\frac{v^2}{2L't}}\sinh v\sin\frac{\pi v}{L't}}{\sqrt{1+\frac{2\cosh v}{c}+\frac{1}{c^2}}}\,dv
$$
with constants
\begin{align*}
L'&:=\int\limits_\mbR \vf'^2(q)\,dq>0\\
\intertext{and}
L''&:=\int\limits_\mbR \vf''^2(q)\,dq>0,
\end{align*}
and also to establish the asymptotic equality
\begin{equation}
\label{equation6}
\overline{\mu}_t(c)=\dfrac{e^{-\frac{L't}{8}}\sqrt{L''}}{\pi\sqrt{2L'}} \cdot \sqrt{\dfrac{c}{\ln c}} \cdot \exp\left[-\frac{(\ln c)^2}{2L't}\right] \cdot (1+\overline{o}(1)),\quad c\to +\infty.
\end{equation}

The structure of the main part of this paper is as follows. In Section~\ref{section2} we prove the stationarity of the stochastic process $\{x(u,t)-u,\; u\in\mbR\}$ and $\theta$-homogeneity (see Definition~\ref{definition2.5} below) of the stochastic process $\{(p_t(u),p'_t(u)),\; u\in\mbR\}$ (here the usual derivative is considered), in Section~\ref{section3} we find the density of the joint distribution of the random variables $p_t(u)$ and $p'_t(u)$, and in Section~\ref{section4} we establish equalities~\eqref{equation5} and~\eqref{equation6}.

\section{Stationarity and $\theta$-homogeneity}
\label{section2}

To prove the stationarity (for fixed $t$) of the stochastic process $\{x(u,t)-u,\; u\in\mbR\}$ we will use the version of the Euler--Maruyama discrete approximation scheme presented in~\cite{Nishchenko}. For $t>0$ and $n\geqslant 1$ set
\begin{gather*}
x_0^n(u,t)\equiv u,\quad u\in\mbR,\\
x_{k+1}^n(u,t):=x_k^n(u,t)+\xi_{k+1}^n(x_k^n(u,t),t),\quad u\in\mbR,\quad 0\leqslant k\leqslant n-1,
\end{gather*}
where
$$
\xi_{k+1}^n(u,t):=\int\limits_{\frac{kt}{n}}^{\frac{(k+1)t}{n}} \int\limits_\mathbb{R} \varphi(u-q)\,W(dq,ds),\quad u\in\mathbb{R},\quad 0\leqslant k\leqslant n-1.
$$
By Theorem~4 of paper~\cite{Nishchenko} the stochastic process $\{x_n^n(u,1),\; u\in\mbR\}$ defined recurrently in this way approximates the stochastic process $\{x(u,1),\; u\in\mbR\}$ in the uniform metric on the interval $[0;1]$. However, the presented proof of this theorem can be easily extended \emph{\foreignlanguage{latin}{mutatis mutandis}} to the case of arbitrary time $t>0$ and interval $[a;b]$, and so its corresponding generalization can be formulated in the following (slightly simplified, but sufficient for our purposes) form.

\begin{theorem}
\label{theorem2.1}
For any time $t>0$ and interval $[a;b]\subset\mbR$ there exists a constant $C>0$ depending on the time $t$, the length of the interval $[a;b]$ and the function $\vf$, such that for all $n\geqslant 1$
$$
\E\norm{x_n^n(\cdot,t)-x(\cdot,t)}_{[a;b]}\leqslant\dfrac{C}{\sqrt{n}},
$$
where $\norm{\cdot}_{[a;b]}$ is the supremum-norm on the interval $[a;b]$.
\end{theorem}

\begin{theorem}
\label{theorem2.2}
For any $t\geqslant 0$ the stochastic process $\{x(u,t)-u,\; u\in\mbR\}$ is (strictly) stationary.
\end{theorem}

\begin{proof}
Fix arbitrary $t>0$ (in the case of $t=0$ there is nothing to prove).

With the help of the principle of mathematical induction, using characteristic functions and taking into account that $\xi_{k+1}^n$ is independent of $\mcF_{\frac{kt}{n}}$ for every $k\in\{0,\ldots,n-1\}$, where $(\mcF_s,\; s\geqslant 0)$ is the filtration generated by the Wiener sheet $W$ (and completed with the sets of zero probability), one can show that all stochastic processes $\{x_k^n(u,t)-u,\; u\in\mbR\}$, $0\leqslant k\leqslant n$, are stationary.

To prove the stationarity of the stochastic process $\{x(u,t)-u,\; u\in\mbR\}$ note that from Theorem~\ref{theorem2.1} it follows that for any $u_1,\ldots,u_m\in\mbR$
$$
\E\max_{1\leqslant k\leqslant m} \abs{x_n^n(u_k,t)-x(u_k,t)}\rightarrow 0,\quad n\to\infty,
$$
and so
\begin{gather*}
(x_n^n(u_1,t)-u_1,\ldots,x_n^n(u_m,t)-u_m)\stackrel{w}{\longrightarrow}\\
\stackrel{w}{\longrightarrow} (x(u_1,t)-u_1,\ldots,x(u_m,t)-u_m),\quad n\to\infty.
\end{gather*}
Therefore, for arbitrary $h>0$ and any Borel sets $\Delta_1,\ldots,\Delta_m\in \mcB(\mbR)$ we have
\begin{gather*}
\Prob{[x(u_1+h,t)-(u_1+h)]\in\Delta_1,\ldots,[x(u_m+h,t)-(u_m+h)]\in\Delta_m}=\\
=\lim_{n\to\infty} \Prob{[x_n^n(u_1+h,t)-(u_1+h)]\in\Delta_1,\ldots, [x_n^n(u_m+h,t)-(u_m+h)]\in\Delta_m}=\\
=\lim_{n\to\infty} \Prob{[x_n^n(u_1,t)-u_1]\in\Delta_1,\ldots, [x_n^n(u_m,t)-u_m]\in\Delta_m}=\\
=\Prob{[x(u_1,t)-u_1]\in\Delta_1,\ldots,[x(u_m,t)-u_m]\in\Delta_m}.
\end{gather*}
Thus, the distributions of the random vectors
\begin{gather*}
(x(u_1+h,t)-(u_1+h),\ldots,x(u_m+h,t)-(u_m+h))\\
\intertext{and}
(x(u_1,t)-u_1,\ldots,x(u_m,t)-u_m)
\end{gather*}
coincide. The theorem is proved.
\end{proof}

\begin{corollary}
\label{corollary2.3}
For any $t\geqslant 0$ the three-dimensional stochastic process
$$
\left\lbrace\left(x(u,t)-u,\dfrac{\partial x}{\partial u}(u,t),\dfrac{\partial^2 x}{\partial u^2}(u,t)\right),\; u\in\mbR\right\rbrace
$$
is stationary.
\end{corollary}

Now we will need several definitions (see~\cite[Chapter~5]{Zirbel}).

\begin{definition}
\label{definition2.4}
A family of mappings $\{\theta_z\colon\Omega \rightarrow \Omega,\; z\in\mbR\}$ is called a \emph{family of (spatial) shifts} if
\begin{enumerate}
\item[(i)]
the mapping $\mbR \times \Omega\ni (z,\omega) \mapsto \theta_z(\omega)\in \Omega$ is measurable;

\item[(ii)]
$\theta_y \circ \theta_z=\theta_{y+z}$ for any $y,z\in\mbR$;

\item[(iii)]
$\theta_0$ is the identity mapping;

\item[(iv)]
$\mathbf{P} \circ \theta_z^{-1}=\mathbf{P}$ for any $z\in\mbR$.
\end{enumerate}
\end{definition}

\begin{definition}
\label{definition2.5}
A stochastic process $\{\xi(u),\; u\in\mbR\}$ is called a \emph{$\theta$-homogeneous stochastic process} if
$$
\forall\, \omega\in\Omega\quad \forall\, u,z\in\mbR:\quad \xi(\theta_z\omega,u)=\xi(\omega,u+z).
$$
\end{definition}

\begin{definition}
A mapping $G\colon\Omega \times \mbR \rightarrow \mbR$ is called a \emph{$\theta$-homogeneous random transformation} if
$$
\forall\, \omega\in\Omega\quad \forall\, u,z\in\mbR:\quad G(\theta_z\omega,u)=G(\omega,u+z)-z.
$$
\end{definition}

\begin{definition}
A mapping $A\colon\Omega \times \mbR \times \mbR_+ \rightarrow E$ taking values in some measurable space $(E,\mathcal{E})$ is called a \emph{$\theta$-homogeneous random field} if it is measurable and
$$
\forall\, \omega\in\Omega\quad \forall\, u,z\in\mbR\quad \forall\, t\geqslant 0:\quad A(\theta_z\omega,u,t)=A(\omega,u+z,t).
$$
\end{definition}

\begin{definition}
A stochastic flow $\{\vf_{s,t}(u),\; u\in\mbR,\; 0\leqslant s\leqslant t<\infty\}$ is called a \emph{$\theta$-homogeneous stochastic flow} if
$$
\forall\, \omega\in\Omega\quad \forall\, u\in\mbR\quad \forall\, s,t\geqslant 0,\; s\leqslant t:\quad \vf_{s,t}(\theta_z\omega,u)=\vf_{s,t}(\omega,u+z)-z.
$$
\end{definition}

\begin{lemma}
\label{lemma2.9}
For any $t\geqslant 0$ the stochastic processes
\begin{gather*}
\left\lbrace\left(x(u,t)-u,\dfrac{\partial x}{\partial u}(u,t),\dfrac{\partial^2 x}{\partial u^2}(u,t)\right),\; u\in\mbR\right\rbrace
\intertext{and}
\{(p_t(u),p'_t(u)),\; u\in\mbR\}
\end{gather*}
are $\theta$-homogeneous.
\end{lemma}

\begin{proof}
For the proof it is enough to note that the canonical representation of the stochastic process
$$
\left\lbrace\left(x(u,t)-u,\dfrac{\partial x}{\partial u}(u,t),\dfrac{\partial^2 x}{\partial u^2}(u,t)\right),\; u\in\mbR\right\rbrace
$$
on the space
$$
\Omega=\{\omega=(\omega_1,\omega_2,\omega_3) \mid \omega_1\in C^2(\mathbb{R}),\; \omega_1'>-1,\; \omega_2=\omega_1'+1,\; \omega_3=\omega_1''\}
$$
with the Borel $\sigma$-algebra generated by the uniform metric is $\theta$-homogeneous with respect to the standard spatial shifts (for these shifts conditions~(ii) and~(iii) of Definition~\ref{definition2.4} are obviously satisfied, condition~(i) follows from their continuity, and condition~(iv) follows from Corollary~\ref{corollary2.3}). After this one can directly verify the $\theta$-homogeneity of the stochastic process $\{x^{-1}(u,t)-u,\; u\in\mbR\}$ and then that of the stochastic process $\{(p_t(u),p'_t(u)),\; u\in\mbR\}$.
\end{proof}

The main results of the third section of this paper are based on the following theorem.

\begin{theorem}
\label{theorem2.10}
\cite[Chapter~5, Theorem~4.11]{Zirbel}
Let $\{\vf_{s,t}\colon\mbR \rightarrow \mbR,\; 0\leqslant s\leqslant t<+\infty\}$ be a $\theta$-homogeneous stochastic flow of $C^1$-diffeomorphisms and $A\colon\Omega \times \mbR \times \mbR_+ \rightarrow E$ be a $\theta$-homogeneous random field taking values in some measurable space $(E,\mathcal{E})$. Then for any $s,t\geqslant 0$, $s\leqslant t$, the following propositions hold true:
\begin{enumerate}
\item[\textup{(i)}]
the random transformation $\vf_{s,t}^{-1}(\cdot)$ is a $\theta$-homogeneous diffeomorphism;

\item[\textup{(ii)}]
the stochastic process $\frac{\partial \vf_{s,t}}{\partial u}(\cdot)$ is $\theta$-homogeneous;

\item[\textup{(iii)}]
the stochastic process $\frac{\partial \vf_{s,t}}{\partial u}(\vf_{s,t}^{-1}(\cdot))$ is $\theta$-homogeneous;

\item[\textup{(iv)}]
for any $\mathcal{E}$-measurable function $f\colon E \rightarrow \mbR_+$ and any $u\in\mbR$
$$
\E f(A(\vf_{s,t}(u),t))=\E\left[f(A(0,t)) \cdot \dfrac{1}{\dfrac{\partial \vf_{s,t}}{\partial u}(\vf_{s,t}^{-1}(0))}\right],
$$
provided the expectation on the right-hand side is well-defined and finite.
\end{enumerate}
\end{theorem}

\section{The joint distribution of $p_t(u)$ and $p'_t(u)$}
\label{section3}

In this section we show that the joint distribution of the random variables $p_t(u)$ and $p'_t(u)$ has a density and find its form.

To begin with, we prove the following lemma.

\begin{lemma}
\label{lemma3.1}
The following representation takes place:
\begin{equation}
\label{equation7}
\dfrac{\partial x}{\partial u}(u,t)=\exp\left[-\frac{1}{2}L't+\int\limits_0^t \int\limits_\mbR \vf'(x(u,s)-q)\,W(dq,ds)\right],\quad t\geqslant 0,\quad u\in\mbR.
\end{equation}
Moreover, for any $u\in\mbR$ the stochastic process
\begin{equation}
\label{equation8}
w_u(t):=\dfrac{1}{\sqrt{L'}} \int\limits_0^t \int\limits_\mbR \vf'(x(u,s)-q)\,W(dq,ds),\quad t\geqslant 0,
\end{equation}
is a standard Wiener process.
\end{lemma}

\begin{proof}
The properties of the integral with respect to a Wiener sheet imply that the stochastic process $\{w_u(t),\; t\geqslant 0\}$ defined by~\eqref{equation8} is a continuous $(\mcF_t)$-adapted martingale with the quadratic variation
$$
\qv{w_u}_t=\dfrac{1}{L'} \int\limits_0^t \int\limits_\mbR \vf'^2(x(u,s)-q)\,dq\,ds=t,\quad t\geqslant 0,
$$
and so, by L\'evy's characterizing theorem it is a standard Brownian motion.

Furthermore, differentiating both sides of the equation for $x(u,t)$ (see~\cite[Theorem~3.3.3]{Kunita}), we obtain that
\begin{equation}
\label{equation9}
\dfrac{\partial x}{\partial u}(u,t)=1+\int\limits_0^t \int\limits_\mbR \vf'(x(u,s)-q) \dfrac{\partial x}{\partial u}(u,s)\,W(dq,ds),\quad t\geqslant 0.
\end{equation}
Using It\^{o}'s formula, it is easy to show that the stochastic process defined by the right-hand side of~\eqref{equation7} satisfies the stochastic integral equation~\eqref{equation9} (with respect to the unknown stochastic process $\{\frac{\partial x}{\partial u}(u,t),\; t\geqslant 0\}$). Therefore, equation~\eqref{equation7} is now implied by the uniqueness of the strong solution of this equation.
\end{proof}

\begin{corollary}
\label{corollary3.2}
For any $t>0$ and $u\in\mbR$ the distribution of the random variable $\frac{\partial x}{\partial u}(u,t)$ has a density of the form
$$
\pi\left[\dfrac{\partial x}{\partial u}(u,t)\right](z)=\dfrac{1}{\sqrt{2\pi L't}} \cdot \exp\left[-\dfrac{\left(\ln z+\frac{3}{2}L't\right)^2}{2L't}+L't\right],\quad z>0.
$$
\end{corollary}

\begin{theorem}
\label{theorem3.3}
For any $t>0$ and $u\in\mbR$ the distribution of the random variable $p_t(u)$ has a density of the form
\begin{equation}
\label{equation10}
\pi[p_t(u)](z)=\dfrac{1}{\sqrt{2\pi L't}} \cdot \exp\left[-\dfrac{\left(\ln z+ \frac{3}{2}L't\right)^2}{2L't}+L't\right],\quad z>0.
\end{equation}
\end{theorem}

\begin{proof}
Set
\begin{gather*}
E:=(0;+\infty),\\
\mathcal{E}:=\mcB((0;+\infty)),
\end{gather*}
and, for fixed $t_0>0$,
$$
A(u,t):=p_{t_0}(u),\quad u\in\mbR,\quad t\geqslant 0.
$$
Note that the mapping $(\omega,u,t)\mapsto A(\omega,u,t)$ is measurable, since such is the mapping $(\omega,u,t)\mapsto x(\omega,u,t)$.

Therefore, for the stochastic flow~\eqref{equation3} and the function
$$
f(z):=\dfrac{1}{z} \cdot \1\{z\in\Delta\},\quad z\in E,
$$
where the set $\Delta\in\mathcal{E}$ is arbitrary, by Theorem~\ref{theorem2.10} with $s=0$ and $t=t_0$ we have (for convenience, until the end of the proof we write $t$ instead of $t_0$)
$$
\E\left[\dfrac{1}{p_t(x(u,t))} \cdot \1\left\lbrace p_t(x(u,t))\in\Delta\right\rbrace\right]=\E\left[\dfrac{1}{p_t(0)} \cdot \1\left\lbrace p_t(0)\in\Delta\right\rbrace \cdot p_t(0)\right],
$$
so that
\begin{equation}
\label{equation11}
\Prob{p_t(0)\in\Delta}=\E\left[\dfrac{\partial x}{\partial u}(u,t) \cdot \1\left\lbrace\dfrac{1}{\frac{\partial x}{\partial u}(u,t)}\in\Delta\right\rbrace\right].
\end{equation}
However, from Corollary~\ref{corollary3.2} it follows that
\begin{equation}
\label{equation12}
\begin{split}
&\E\left[\dfrac{\partial x}{\partial u}(u,t) \cdot \1\left\lbrace\dfrac{1}{\frac{\partial x}{\partial u}(u,t)}\in\Delta\right\rbrace\right]=\\
=&\int\limits_\Delta \dfrac{1}{\sqrt{2\pi L't}} \cdot \exp\left[-\frac{\left(\ln z+ \frac{3}{2}L't\right)^2}{2L't}+L't\right]\,dz.
\end{split}
\end{equation}
Since the set $\Delta$ is arbitrary, equalities~\eqref{equation11} and~\eqref{equation12} imply the required assertion.
\end{proof}

\begin{remark}
Thus, for any $t>0$ and $u\in\mbR$ the following equality in distribution holds:
$$
p_t(u)\stackrel{d}{=}\dfrac{\partial x}{\partial u}(u,t).
$$
\end{remark}

\begin{corollary}
For any $t>0$ and $u\in\mbR$ the following equality holds:
$$
\Prob{p_t(u)>c}=\dfrac{e^{-\frac{L't}{8}}\sqrt{L't}}{\sqrt{2\pi}} \cdot \dfrac{1}{\sqrt{c} \cdot \ln c} \cdot \exp\left[-\frac{(\ln c)^2}{2L't}\right] \cdot (1+\overline{o}(1)),\quad c\to +\infty.
$$
\end{corollary}

\begin{proof}
To prove this equality one should re-write its left-hand side in the form of an integral of the density of the distribution of the random variable $p_t(u)$ from Formula~\eqref{equation10}, to reduce it by the change-of-variable formula to an integral of the density $\mathfrak{p}$ of the standard Gaussian distribution, and finally to use the well-known (e.~g., see~\cite[Chapter~1]{Lifshits}) relation
\[
\int\limits_c^{+\infty} \mathfrak{p}(u)\,du=\dfrac{1}{c} \cdot \mathfrak{p}(c) \cdot (1+\overline{o}(1)),\quad c\to +\infty.
\qedhere
\]
\end{proof}

To prove the next result we will need an auxiliary lemma concerning the conditional expectation of the It\^{o} integral with respect to a Wiener process.

\begin{lemma}
\label{lemma3.6}
Let a continuous stochastic process $\{\xi_t,\; t\geqslant 0\}$ be independent of a Wiener process $\{\beta_t,\; t\geqslant 0\}$ and the It\^{o} integral
$$
\int\limits_0^t \xi_s\,d\beta_s,\quad t\geqslant 0,
$$
be well-defined. Then for any Borel set $\Delta\subset\mbR$ the following equality holds:
\begin{gather*}
\E\left(\left.\1\left\lbrace\int\limits_0^t \xi_s\,d\beta_s\in\Delta\right\rbrace \right|\xi\right)=\\
=\int\limits_\Delta \dfrac{1}{\sqrt{2\pi}\sigma_t} e^{-\frac{v^2}{2\sigma_t^2}}\,dv \cdot \1\left\lbrace\sigma_t>0\right\rbrace+\1\left\lbrace 0\in\Delta\right\rbrace \cdot \1\left\lbrace\sigma_t=0\right\rbrace,
\end{gather*}
where
$$
\sigma_t:=\left(\int\limits_0^t \xi_s^2\,ds\right)^{1/2}.
$$
\end{lemma}

\begin{proof}
The proof of this lemma can be obtained in a standard way by approximating indicator functions with smooth bounded functions and is therefore omitted.
\end{proof}

Now fix arbitrary $u\in\mbR$ and define the stochastic processes
\begin{align*}
X_1(t)&:=\dfrac{\partial x}{\partial u}(u,t),\quad t\geqslant 0,\\
X_2(t)&:=\dfrac{\frac{\partial^2 x}{\partial u^2}(u,t)}{\frac{\partial x}{\partial u}(u,t)},\quad t\geqslant 0.
\end{align*}

\begin{lemma}
\label{lemma3.7}
For any $t>0$ the joint distribution of the random variables $X_1(t)$ and $X_2(t)$ has a density of the form
\begin{gather*}
\pi[X_1(t),X_2(t)](z_1,z_2)=\\
=\dfrac{e^{\frac{\pi^2}{2L't}-\frac{L't}{8}}}{\pi\sqrt{2\pi L''t}} \cdot \dfrac{1}{\sqrt{z_1}} \cdot \int\limits_0^{+\infty} \dfrac{e^{-\frac{v^2}{2L't}} \sinh v\sin\frac{\pi v}{L't}}{\left((1+2z_1\cosh v+z_1^2)+ \frac{L'}{L''}z_2^2\right)^{3/2}}\,dv,\quad z_1>0,\quad z_2\in\mbR.
\end{gather*}
\end{lemma}

\begin{proof}
Note that from representation~\eqref{equation7} it follows that
\begin{equation}
\label{equation13}
\dfrac{\partial^2 x}{\partial u^2}(u,t)=\dfrac{\partial x}{\partial u}(u,t) \cdot \int\limits_0^t \int\limits_\mbR \vf''(x(u,s)-q)\dfrac{\partial x} {\partial u}(u,s)\,W(dq,ds),\quad t\geqslant 0.
\end{equation}
From~\eqref{equation9} and~\eqref{equation13} we obtain that
\begin{align*}
X_1(t)&=1+\int\limits_0^t \int\limits_\mbR \vf'(x(u,s)-q) X_1(s)\,W(dq,ds),\quad t\geqslant 0,\\
X_2(t)&=\int\limits_0^t \int\limits_\mbR \vf''(x(u,s)-q) X_1(s)\,W(dq,ds),\quad t\geqslant 0.
\end{align*}
By the properties of the integral with respect to a Wiener sheet, we have
\begin{align*}
\qv{X_1}_t&=\int\limits_0^t \int\limits_\mbR \vf'^2(x(u,s)-q)X_1^2(s)\,dq\,ds=L' \cdot \int\limits_0^t X_1^2(s)\,ds,\quad t\geqslant 0,\\
\qv{X_2}_t&=\int\limits_0^t \int\limits_\mbR \vf''^2(x(u,s)-q)X_1^2(s)\,dq\,ds=L'' \cdot \int\limits_0^t X_1^2(s)\,ds,\quad t\geqslant 0,\\
\jqv{X_1}{X_2}_t&=\int\limits_0^t \int\limits_\mbR \vf'(x(u,s)-q)\varphi''(x(u,s)-q)X_1^2(s)\,dq\,ds=0,\quad t\geqslant 0.
\end{align*}
Since with probability one $X_1(t)>0$ for all $t\geqslant 0$, using Doob's theorem (see~\cite[Chapter~5, Theorem~5.12]{LiptserShiryaev}), we obtain that the pair of the stochastic processes $X_1$ and $X_2$ is a solution of the following system of equations:
\begin{equation}
\label{equation14}
\begin{cases}
X_1(t)=1+\sqrt{L'} \cdot \int\limits_0^t X_1(s)\,dW_1(s),\quad t\geqslant 0,\\
X_2(t)=\sqrt{L''} \cdot \int\limits_0^t X_1(s)\,dW_2(s),\quad t\geqslant 0,
\end{cases}
\end{equation}
where the Wiener processes $\{W_1(t),\; t\geqslant 0\}$ and $\{W_2(t),\; t\geqslant 0\}$ can be defined by the equalities
\begin{align*}
W_1(t)&=\dfrac{1}{\sqrt{L'}} \int\limits_0^t \dfrac{dX_1(s)}{X_1(s)},\quad t\geqslant 0,\\
W_2(t)&=\dfrac{1}{\sqrt{L''}} \int\limits_0^t \dfrac{dX_2(s)}{X_1(s)},\quad t\geqslant 0.
\end{align*}
Note that the Wiener processes $W_1$ and $W_2$ are independent, since
$$
\jqv{W_1}{W_2}_t=\dfrac{1}{\sqrt{L'L''}} \int\limits_0^t \dfrac{d\jqv{X_1}{X_2}_s} {X_1^2(s)}=0,\quad t\geqslant 0.
$$

Solving system~\eqref{equation14}, we get the representations
\begin{gather*}
X_1(t)=\exp\left[-\frac{1}{2}L't+\sqrt{L'}W_1(t)\right],\quad t\geqslant 0,\\
X_2(t)=\sqrt{L''} \int\limits_0^t \exp\left[-\frac{1}{2}L's+ \sqrt{L'}W_1(s)\right]dW_2(s),\quad t\geqslant 0.
\end{gather*}

Now to find the density of the joint distribution of the random variables $X_1(t)$ and $X_2(t)$ for $t>0$ note that for any Borel sets $\Delta_1\subset(0;+\infty)$ and $\Delta_2\subset\mbR$ we have
$$
\Prob{X_1(t)\in\Delta_1,\; X_2(t)\in\Delta_2}=\E\left[\1\{X_1(t)\in\Delta_1\} \cdot \E\left(\1\{X_2(t)\in\Delta_2\} \mid X_1\right)\right].
$$
However, since as is easily seen the filtrations generated by the stochastic processes $X_1$ and $W_1$ (and completed with the sets of zero probability) coincide and $W_1$ and $W_2$ are independent, $X_1$ and $W_2$ are independent too, and so by Lemma~\ref{lemma3.6}
\begin{gather*}
\E\left(\left.\1\{X_2(t)\in\Delta_2\}\right| X_1\right)=\\
=\E\left(\left.\1\left\lbrace\sqrt{L''}\int\limits_0^t X_1(s)\,dW_2(s)\in\Delta_2\right\rbrace\right| X_1\right)=\\
=\frac{1}{\sqrt{L''}}\int\limits_{\Delta_2} \dfrac{1}{\sqrt{2\pi}\sigma_t} e^{-\frac{v^2}{2L''\sigma_t^2}}\,dv,
\end{gather*}
where
$$
\sigma_t:=\left(\int\limits_0^t X_1^2(s)\,ds\right)^{1/2}>0.
$$
Therefore,
\begin{gather*}
\Prob{X_1(t)\in\Delta_1,\; X_2(t)\in\Delta_2}=\\
=\E\left(\1\{X_1(t)\in\Delta_1\} \cdot \frac{1}{\sqrt{L''}}\int\limits_{\Delta_2} \dfrac{1}{\sqrt{2\pi}\sigma_t}e^{-\frac{v^2}{2L''\sigma_t^2}}\,dv\right)=\\
=\frac{1}{\sqrt{L''}}\int\limits_{\Delta_2} \E\left(\1\{X_1(t)\in\Delta_1\} \cdot \dfrac{1}{\sqrt{2\pi}\sigma_t}e^{-\frac{v^2}{2L''\sigma_t^2}}\right)\,dv.
\end{gather*}
The density of the joint distribution of the random variables $X_1(t)$ and $\sigma_t^2\equiv\int\limits_0^t X_1^2(s)\,ds$ has the following form (see~\cite[p.~265, Formula~1.10.8]{BorodinSalminen}):
$$
\pi\left[X_1(t),\sigma_t^2\right](z_1,z_2)=\dfrac{\exp\left[-\frac{L't}{8}- \frac{1+z_1^2}{2L'z_2}\right]}{2z_1^{3/2}z_2} \cdot i_{\frac{L't}{2}}\left(\dfrac{z_1}{L'z_2}\right),\quad z_1,z_2>0,
$$
where (see~\cite[p.~644]{BorodinSalminen})
$$
i_y(z):=\dfrac{ze^{\frac{\pi^2}{4y}}}{\pi\sqrt{\pi y}}\int\limits_0^{+\infty} \exp\left[-z\cosh v- \frac{v^2}{4y}\right]\sinh v\sin\dfrac{\pi v}{2y}\,dv,\quad y,z>0.
$$
Hence (in the second equality we simply change $v$ to $z_2$ and $z_2$ to $v$)
\begin{gather*}
\Prob{X_1(t)\in\Delta_1,\; X_2(t)\in\Delta_2}=\\
=\int\limits_{\Delta_2}dv \int\limits_{\Delta_1}dz_1 \int\limits_0^{+\infty}dz_2 \left[\dfrac{\exp\left[-\frac{v^2}{2L''z_2}- \frac{L't}{8}-\frac{1+z_1^2}{2L'z_2}\right]}{2\sqrt{2\pi L''} \cdot z_1^{3/2} z_2^{3/2}} \cdot i_{\frac{L't}{2}}\left(\dfrac{z_1}{L'z_2}\right)\right]=\\
=\int\limits_{\Delta_2}dz_2 \int\limits_{\Delta_1}dz_1 \int\limits_0^{+\infty}dv \left[\dfrac{\exp\left[-\frac{z_2^2}{2L''v}-\frac{L't}{8}- \frac{1+z_1^2}{2L'v}\right]}{2\sqrt{2\pi L''} \cdot z_1^{3/2}v^{3/2}} \cdot i_{\frac{L't}{2}}\left(\dfrac{z_1}{L'v}\right)\right]=\\
=\int\limits_{\Delta_1} \int\limits_{\Delta_2} \left[\int\limits_0^{+\infty} \dfrac{\exp\left[-\frac{z_2^2}{2L''v}-\frac{L't}{8}-\frac{1+z_1^2}{2L'v}\right]} {2\sqrt{2\pi L''} \cdot z_1^{3/2}v^{3/2}} \cdot i_{\frac{L't}{2}}\left(\dfrac{z_1}{L'v}\right)\,dv\right]\,dz_1\,dz_2.
\end{gather*}
Thus, for $z_1>0$ and $z_2\in\mbR$ we have
\begin{gather*}
\pi[X_1(t),X_2(t)](z_1,z_2)=\int\limits_0^{+\infty} \dfrac{\exp\left[-\frac{z_2^2}{2L''v}-\frac{L't}{8}-\frac{1+z_1^2}{2L'v}\right]} {2\sqrt{2\pi L''} \cdot z_1^{3/2}v^{3/2}} \cdot i_{\frac{L't}{2}}\left(\dfrac{z_1}{L'v}\right)\,dv=\\
=\dfrac{e^{\frac{\pi^2}{2L't}-\frac{L't}{8}}}{2\pi^2L'\sqrt{L'L''t}} \int\limits_0^{+\infty} \dfrac{e^{-\frac{z_2^2}{2L''v}-\frac{1+z_1^2}{2L'v}}}{\sqrt{z_1} \cdot v^{5/2}} \cdot \left(\int\limits_0^{+\infty} e^{-\frac{z_1}{L'v}\cosh r-\frac{r^2}{2L't}} \sinh r\sin\dfrac{\pi r}{L't}\,dr\right)\,dv=\\
=\dfrac{e^{\frac{\pi^2}{2L't}-\frac{L't}{8}}}{2\pi^2L'\sqrt{L'L''t}} \cdot \dfrac{1}{\sqrt{z_1}} \cdot \int\limits_0^{+\infty} \left[\int\limits_0^{+\infty} \dfrac{\exp\left[-\frac{1}{2v} \left(\frac{z_2^2}{L''}+\frac{1+2z_1\cosh r+z_1^2}{L'}\right)\right]}{v^{5/2}}\,dv\right] \cdot e^{-\frac{r^2}{2L't}}\sinh r \sin\dfrac{\pi r}{L't}\,dr.
\end{gather*}
Finally, since for $K>0$
$$
\int\limits_0^{+\infty} \dfrac{1}{v^{5/2}}e^{-\frac{K}{2v}}\,dv= \dfrac{\sqrt{2\pi}}{K^{3/2}},
$$
we obtain
\begin{gather*}
\pi[X_1(t),X_2(t)](z_1,z_2)=\\
=\dfrac{e^{\frac{\pi^2}{2L't}-\frac{L't}{8}}}{2\pi^2L'\sqrt{L'L''t}} \cdot \dfrac{1}{\sqrt{z_1}} \cdot \int\limits_0^{+\infty} \dfrac{e^{-\frac{v^2}{2L't}} \sinh v\sin\dfrac{\pi v}{L't}}{\left(\frac{z_2^2}{L''}+\frac{1+2z_1\cosh v+z_1^2}{L'}\right)^{3/2}} \cdot \sqrt{2\pi}\,dv=\\
=\dfrac{e^{\frac{\pi^2}{2L't}-\frac{L't}{8}}} {\pi\sqrt{2\pi L''t}} \cdot \dfrac{1}{\sqrt{z_1}} \cdot \int\limits_0^{+\infty} \dfrac{e^{-\frac{v^2}{2L't}} \sinh v\sin\frac{\pi v}{L't}}{\left((1+2z_1\cosh v+z_1^2)+ \frac{L'}{L''}z_2^2\right)^{3/2}}\,dv.
\qedhere
\end{gather*}
\end{proof}

\begin{remark}
\label{remark3.8}
The density of the joint distribution of $X_1(t)$ and $X_2(t)$ is symmetric in the second variable:
$$
\pi[X_1(t),X_2(t)](z_1,-z_2)=\pi[X_1(t),X_2(t)](z_1,z_2),\quad z_1>0,\quad z_2\in\mbR.
$$
\end{remark}

\begin{theorem}
\label{theorem3.9}
For any $t>0$ and $u\in\mbR$ the joint distribution of the random variables $p_t(u)$ and $p'_t(u)$ has a density of the form
\begin{gather*}
\pi\left[p_t(u),p'_t(u)\right](z_1,z_2)=\dfrac{e^{\frac{\pi^2}{2L't}-\frac{L't}{8}}}{\pi\sqrt{2\pi L''t}} \cdot \int\limits_0^{+\infty} \dfrac{z_1^{3/2}e^{-\frac{v^2}{2L't}}\sinh v\sin\frac{\pi v}{L't}} {\left((z_1^2+2z_1^3\cosh v+z_1^4)+\frac{L'}{L''}z_2^2\right)^{3/2}}\,dv,\\
z_1\geqslant 0,\quad z_2\in\mbR,\quad z_1^2+z_2^2\neq 0.
\end{gather*}
\end{theorem}

\begin{proof}
The proof of this theorem is similar to that of Theorem~\ref{theorem3.3}. Set
\begin{gather*}
E:=(0;+\infty) \times \mbR,\\
\mathcal{E}:=\mcB((0;+\infty) \times \mbR),
\end{gather*}
and, for fixed $t_0>0$,
$$
A(u,t):=(p_{t_0}(u),p'_{t_0}(u)),\quad u\in\mbR,\quad t\geqslant 0.
$$
Note that the mapping $(\omega,u,t)\mapsto A(\omega,u,t)$ is measurable, since such is the mapping $(\omega,u,t)\mapsto x(\omega,u,t)$.

Therefore, for the stochastic flow~\eqref{equation3} and the function
$$
f(z_1,z_2)=\dfrac{1}{z_1} \cdot \1\{z_1\in\Delta_1,\; z_2\in\Delta_2\},\quad (z_1,z_2)\in E,
$$
where $\Delta_1\in\mcB((0;+\infty))$ and $\Delta_2\in\mcB(\mbR)$ are arbitrary, by Theorem~\ref{theorem2.10} with $s=0$ and $t=t_0$ we have (for convenience, until the end of the proof we write $t$ instead of $t_0$)
\begin{gather*}
\E\left[\dfrac{\partial x}{\partial u}(u,t) \cdot \1\left\lbrace\dfrac{1}{\frac{\partial x}{\partial u}(u,t)}\in\Delta_1,\; -\dfrac{\frac{\partial^2 x}{\partial u^2}(u,t)}{\left(\frac{\partial x}{\partial u}(u,t)\right)^3}\in\Delta_2\right\rbrace\right]=\\
=\E\left[\dfrac{1}{p_t(0)} \cdot \1\left\lbrace p_t(0)\in\Delta_1,\; p'_t(0)\in\Delta_2\right\rbrace \cdot p_t(0)\right],
\end{gather*}
so that
\begin{gather*}
\Prob{p_t(0)\in\Delta_1,\; p'_t(0)\in\Delta_2}=\E\left[X_1(t) \cdot \1\left\lbrace\dfrac{1}{X_1(t)}\in\Delta_1,\; -\dfrac{X_2(t)}{X_1^2(t)}\in\Delta_2\right\rbrace\right]=\\
=\int\limits_0^{+\infty} \int\limits_\mathbb{R} v_1 \cdot \1\left\lbrace\dfrac{1}{v_1}\in\Delta_1,\; -\dfrac{v_2}{v_1^2}\in\Delta_2\right\rbrace \cdot \pi[X_1(t),X_2(t)](v_1,v_2)\,dv_1\,dv_2=\\
=\left[
\begin{matrix}
z_1=\frac{1}{v_1}\\
z_2=-\frac{v_2}{v_1^2}\\
\end{matrix}
\right]=
\int\limits_{\Delta_1} \int\limits_{\Delta_2} \dfrac{1}{z_1^5} \cdot \pi[X_1(t),X_2(t)]\left(\dfrac{1}{z_1},-\dfrac{z_2}{z_1^2}\right)\,dz_1\,dz_2.
\end{gather*}
From this and from Remark~\ref{remark3.8} it follows that the joint distribution of the random variables $p_t(0)$ and $p'_t(0)$ has a density for which the following representation takes place:
\begin{equation}
\label{equation15}
\pi[p_t(0),p'_t(0)](z_1,z_2)=\dfrac{1}{z_1^5} \cdot \pi[X_1(t),X_2(t)]\left(\dfrac{1}{z_1},\dfrac{z_2}{z_1^2}\right), \quad z_1>0,\quad z_2\in\mbR.
\end{equation}
Therefore, by Lemma~\ref{lemma3.7} we obtain that for $z_1>0$ and $z_2\in\mbR$
\begin{gather*}
\pi[p_t(0),p'_t(0)](z_1,z_2)=\dfrac{1}{z_1^5} \cdot \dfrac{e^{\frac{\pi^2}{2L't}-\frac{L't}{8}}} {\pi\sqrt{2\pi L''t}} \cdot \sqrt{z_1} \cdot \int\limits_0^{+\infty} \dfrac{e^{-\frac{v^2}{2L't}}\sinh v\sin\frac{\pi v}{L't}} {\left(\frac{1+2z_1\cosh v+z_1^2}{z_1^2}+\frac{L'}{L''} \cdot \frac{z_2^2}{z_1^4}\right)^{3/2}}\,dv=\\
=\dfrac{e^{\frac{\pi^2}{2L't}-\frac{L't}{8}}} {\pi\sqrt{2\pi L''t}} \cdot z_1^{3/2} \cdot \int\limits_0^{+\infty} \dfrac{e^{-\frac{v^2}{2L't}}\sinh v\sin\frac{\pi v}{L't}}{\left((z_1^2+2z_1^3\cosh v+z_1^4)+\frac{L'}{L''}z_2^2\right)^{3/2}}\,dv.
\end{gather*}

It remains to note that by Lemma~\ref{lemma2.9} the density of the joint distribution of $p_t(u)$ and $p'_t(u)$ does not depend on $u\in\mbR$.
\end{proof}

\begin{remark}
\label{remark3.10}
The density of the joint distribution of $p_t(u)$ and $p'_t(u)$ is symmetric in the second variable:
$$
\pi[p_t(u),p'_t(u)](z_1,-z_2)=\pi[p_t(u),p'_t(u)](z_1,z_2),\quad z_1>0,\quad z_2\in\mbR.
$$
\end{remark}

\section{Level-crossing intensity for the stochastic process $p_t$}
\label{section4}

By Lemma~\ref{lemma2.9} the stochastic process $\{p_t(u),\; u\in\mbR\}$ is strictly stationary, and by Theorem~\ref{theorem3.3} all its one-dimensional distributions are continuous. Hence (see~\cite[pp.~146--147]{LeadbetterLindgrenRootzen}) for any $c\in\mbR$ almost surely it is not identically equal to $c$ on any interval, and so the number $N_t([0;1];c)$ of crossings of the level $c$ by this stochastic process on the interval $[0;1]$ and the corresponding level-crossing intensity $\mu_t(c)$ are well-defined.

\begin{theorem}
For any $t>0$ we have
$$
\mu_t(c)=\dfrac{\sqrt{2L''} \cdot e^{\frac{\pi^2}{2L't}-\frac{L't}{8}}} {\pi L'\sqrt{\pi t}} \cdot \dfrac{1}{\sqrt{c}} \cdot \int\limits_0^{+\infty} \dfrac{e^{-\frac{v^2}{2L't}}\sinh v\sin \frac{\pi v} {L't}} {\sqrt{1+\frac{2\cosh v}{c}+ \frac{1}{c^2}}}\,dv\qquad\text{for a.~e. $c>0$}.
$$
\end{theorem}

\begin{proof}
It is easy to see that the stochastic process $\{p_t(u),\; u\in [0;1]\}$ satisfies the conditions of Theorem~\ref{theorem1.2} and so
\begin{gather*}
\mu_t(c)=\overline{\mu}_t(c)\qquad \text{for a.~e. $c>0$,}
\intertext{where}
\overline{\mu}_t(c):=\int\limits_0^1 \E(|p'_t(u)|\; |\; p_t(u)=c) \cdot \pi[p_t(u)](c)\,du.
\end{gather*}
However, from Theorem~\ref{theorem3.9}, the strict positivity of the density $\pi[p_t(u)](z)$ for $z>0$ implied by Theorem~\ref{theorem3.3}, and Remark~\ref{remark3.10}, it follows that
\begin{gather*}
\int\limits_0^1 \E(|p'_t(u)|\; |\; p_t(u)=c) \cdot \pi[p_t(u)](c)\,du=\\
=\int\limits_0^1 \int\limits_{-\infty}^{\infty} |z| \cdot \pi[p_t(u),p'_t(u)](c,z)\,dz\,du=2\int\limits_0^{+\infty} z \cdot \pi[p_t(0),p'_t(0)](c,z)\,dz
\end{gather*}
and so
\begin{equation}
\label{equation16}
\overline{\mu}_t(c)=2\int\limits_0^{+\infty} z \cdot \pi[p_t(0),p'_t(0)](c,z)\,dz.
\end{equation}
Therefore, by the same Theorem~\ref{theorem3.9}
\begin{gather*}
\overline{\mu}_t(c)=2\int\limits_0^{+\infty} \left[\dfrac{e^{\frac{\pi^2}{2L't}-\frac{L't}{8}}} {\pi\sqrt{2\pi L''t}} \cdot z \cdot \int\limits_0^{+\infty} \dfrac{c^{3/2}e^{-\frac{v^2}{2L't}}\sinh v \sin\frac{\pi v}{L't}}{\left((c^2+2c^3\cosh v+c^4)+ \frac{L'}{L''}z^2\right)^{3/2}}\,dv\right]\,dz=\\
=\dfrac{\sqrt{2} \cdot e^{\frac{\pi^2}{2L't}-\frac{L't}{8}}} {\pi\sqrt{\pi L''t}} \cdot c^{3/2} \cdot \int\limits_0^{+\infty} \left[\int\limits_0^{+\infty} \dfrac{z\,dz}{\left((c^2+2c^3\cosh v+c^4)+\frac{L'}{L''}z^2\right)^{3/2}}\right] e^{-\frac{v^2}{2L't}}\sinh v\sin\frac{\pi v}{L't}\,dv.
\end{gather*}
Finally, since for any $A,B>0$
\begin{gather*}
\int\limits_0^{+\infty} \dfrac{z\,dz}{\left(A+Bz^2\right)^{3/2}}= \dfrac{1}{B\sqrt{A}},
\end{gather*}
we obtain
\begin{align*}
\overline{\mu}_t(c)&=\dfrac{\sqrt{2} \cdot e^{\frac{\pi^2}{2L't}-\frac{L't}{8}}} {\pi\sqrt{\pi L''t}} \cdot c^{3/2} \cdot \int\limits_0^{+\infty} \dfrac{L''e^{-\frac{v^2}{2L't}}\sinh v\sin\frac{\pi v}{L't}} {L'\sqrt{c^2+2c^3\cosh v+c^4}}\,dv=\\
&=\dfrac{\sqrt{2L''} \cdot e^{\frac{\pi^2}{2L't}-\frac{L't}{8}}} {\pi L'\sqrt{\pi t}} \cdot \dfrac{1}{\sqrt{c}} \cdot \int\limits_0^{+\infty} \dfrac{e^{-\frac{v^2}{2L't}}\sinh v\sin \frac{\pi v} {L't}} {\sqrt{1+\frac{2\cosh v}{c}+\frac{1}{c^2}}}\,dv.
\qedhere
\end{align*}
\end{proof}

It seems difficult to find the asymptotics of $\overline{\mu}_t(c)$ as $c\to +\infty$ by direct analytic methods. However, one can find it with the help of a probabilistic approach\footnote{The author is grateful to Prof.~A.~A.~Dorogovtsev for the advice to use this approach.}.

\begin{theorem}
For any $t>0$ we have
$$
\overline{\mu}_t(c)=\dfrac{e^{-\frac{L't}{8}}\sqrt{L''}}{\pi\sqrt{2L'}} \cdot \sqrt{\dfrac{c}{\ln c}} \cdot \exp\left[-\frac{(\ln c)^2}{2L't}\right] \cdot (1+\overline{o}(1)),\quad c\to +\infty.
$$
\end{theorem}

\begin{proof}
From equalities~\eqref{equation16} and~\eqref{equation15} we obtain
\begin{align*}
\overline{\mu}_t(c)&=2\int\limits_0^{+\infty} z \cdot \pi[p_t(0),p'_t(0)](c,z)\,dz=\\
&=\dfrac{2}{c^5} \cdot \int\limits_0^{+\infty} z \cdot \pi[X_1(t),X_2(t)]\left(\frac{1}{c},\dfrac{z}{c^2}\right)\,dz=\\
&=\dfrac{2}{c} \cdot \int\limits_0^{+\infty} z \cdot \pi[X_1(t),X_2(t)]\left(\frac{1}{c},z\right)\,dz.
\end{align*}
Note that since
\begin{gather*}
\int\limits_0^{+\infty} z \cdot \pi[X_1(t),X_2(t)]\left(\frac{1}{c},z\right)\,dz= \pi[X_1(t)]\left(\frac{1}{c}\right) \cdot \E\left((X_2(t))_+\left|X_1(t)=\dfrac{1}{c}\right.\right),
\end{gather*}
where
$$
(z)_+:=
\begin{cases}
z,\quad z\geqslant 0,\\
0,\quad z<0,\\
\end{cases}
$$
we have
\begin{gather*}
\overline{\mu}_t(c)=\dfrac{2}{c} \cdot \pi[X_1(t)] \left(\frac{1}{c}\right) \cdot \E\left((X_2(t))_+\left|X_1(t)=\dfrac{1}{c}\right.\right)=\\
=\dfrac{2}{c} \cdot \pi[X_1(t)] \left(\frac{1}{c}\right) \cdot \E\left(\left.\E\left[\left.\left(\sqrt{L''}\int\limits_0^t X_1(s)\,dW_2(s)\right)_+\right|X_1\right]\right|X_1(t)=\dfrac{1}{c}\right).
\end{gather*}
Furthermore, in a standard way by using the approximation of the stochastic integral with partial sums, the independence of the stochastic process $X_1$ from the Wiener process $W_2$, and the fact that if $\xi\sim\mcN(0;\sigma^2)$, then
$$
\E(\xi)_+=\dfrac{\sigma}{\sqrt{2\pi}},
$$
one can show that
$$
\E\left(\left.\left(\sqrt{L''}\int\limits_0^t X_1(s)\,dW_2(s)\right)_+\right|X_1\right)=\dfrac{\sqrt{L''}}{\sqrt{2\pi}} \cdot \left(\int\limits_0^t X_1^2(s)\,ds\right)^{1/2}.
$$
Therefore,
$$
\overline{\mu}_t(c)=\dfrac{2}{c} \cdot \pi[X_1(t)] \left(\frac{1}{c}\right) \cdot \dfrac{\sqrt{L''}}{\sqrt{2\pi}} \cdot \E\left(\left.\left(\int\limits_0^t X_1^2(s)\,ds\right)^{1/2}\right|X_1(t)=\dfrac{1}{c}\right).
$$
However,
\begin{gather*}
\E\left(\left.\left(\int\limits_0^t X_1^2(s)\,ds\right)^{1/2}\right|X_1(t)= \dfrac{1}{c}\right)=\\
=\E\left(\left.\left(\int\limits_0^t \exp\left[-L's+2\sqrt{L'}\left(\frac{s}{t}W_1(t)+\widetilde{B}(s)\right)\right]\,ds\right)^{1/2}\right|W_1(t)=\dfrac{1}{\sqrt{L'}}\left(\dfrac{1}{2}L't-\ln c\right)\right),
\end{gather*}
where
$$
\widetilde{B}(s):=W_1(s)-\frac{s}{t}W_1(t),\quad 0\leqslant s\leqslant t.
$$
Since the stochastic process $\{\widetilde{B}(s),\; 0\leqslant s\leqslant t\}$ does not depend on the random variable $W_1(t)$, we obtain
\begin{gather*}
\E\left(\left.\left(\int\limits_0^t X_1^2(s)\,ds\right)^{1/2}\right|X_1(t)= \dfrac{1}{c}\right)=\\
=\E\left(\int\limits_0^t \exp\left[-L's+2\sqrt{L'} \cdot \frac{s}{t} \cdot \frac{1}{\sqrt{L'}} \cdot \left(\dfrac{1}{2}L't-\ln c\right)+ 2\sqrt{L'}\widetilde{B}(s)\right]\,ds\right)^{1/2}=\\
=\E\left(\int\limits_0^t \exp\left[-2\ln c \cdot \frac{s}{t}+ 2\sqrt{L'}\widetilde{B}(s)\right]\,ds\right)^{1/2}.
\end{gather*}
Set
$$
B(s):=\dfrac{1}{\sqrt{t}}\widetilde{B}(st),\quad 0\leqslant s\leqslant 1.
$$
Then $\{B(s),\; 0\leqslant s\leqslant 1\}$ is a standard Brownian bridge and
\begin{gather*}
\E\left(\int\limits_0^t \exp\left[-2\ln c \cdot \frac{s}{t}+2\sqrt{L'}\widetilde{B}(s)\right]\,ds\right)^{1/2}=\\
=\E\left(\int\limits_0^t \exp\left[-2\ln c \cdot \frac{s}{t}+2\sqrt{L'} \cdot \sqrt{t} B\left(\frac{s}{t}\right)\right]\,ds\right)^{1/2}=\\
=\sqrt{t} \cdot \E\left(\int\limits_0^1 \exp\left[-2\ln c \cdot s+2\sqrt{L't} \cdot B(s)\right]\,ds\right)^{1/2}.
\end{gather*}

Now define functions $h_c\colon [0;1]\rightarrow\mbR$, $c>1$, by the equality
$$
h_c(s):=\dfrac{2c^2\ln c}{c^2-1} \cdot e^{-2\ln c \cdot s},\quad 0\leqslant s\leqslant 1,\quad c>1,
$$
and note that they satisfy the following conditions:
\begin{enumerate}
\item[1)]
$h_c(s)\geqslant 0$, $s\in [0;1]$, $c>1$;

\item[2)]
$\int\limits_0^1 h_c(s)\,ds=1$, $c>1$;

\item[3)]
$\forall\, \ve\in (0;1):\quad \lim\limits_{c\to +\infty} \int\limits_\ve^1 h_c(s)\,ds=0$.
\end{enumerate}
Therefore,
$$
\lim_{c\to +\infty} \int\limits_0^1 h_c(s)e^{2\sqrt{L't} \cdot B(s)}\,ds=e^{2\sqrt{L't} \cdot B(0)}=1.
$$
Since for any $n\geqslant 1$ we have
\begin{gather*}
\E\sup_{c>1}\left(\int\limits_0^1 h_c(s)e^{2\sqrt{L't} \cdot B(s)}\,ds\right)^n\leqslant \E\sup_{c>1}\left(\int\limits_0^1 h_c(s)\,ds \cdot e^{2\sqrt{L't} \cdot \max\limits_{0\leqslant s\leqslant 1} B(s)}\right)^n=\\
=\E\sup_{c>1} e^{2n\sqrt{L't} \cdot \max\limits_{0\leqslant s\leqslant 1} B(s)}=\E e^{2n\sqrt{L't} \cdot \max\limits_{0\leqslant s\leqslant 1} B(s)}<+\infty,
\end{gather*}
the family of random variables
$$
\left(\int\limits_0^1 h_c(s)e^{2\sqrt{L't} \cdot B(s)}\,ds\right)^{1/2},\quad c>1,
$$
is uniformly integrable, and so
$$
\lim_{c\to +\infty} \E\left(\int\limits_0^1 h_c(s)e^{2\sqrt{L't} \cdot B(s)}\,ds\right)^{1/2}=\E\lim_{c\to +\infty} \left(\int\limits_0^1 h_c(s)e^{2\sqrt{L't} \cdot B(s)}\,ds\right)^{1/2}=1.
$$
Therefore,
$$
\E\left(\int\limits_0^1 \exp\left[-2\ln c \cdot s+2\sqrt{L't} B(s)\right]\,ds\right)^{1/2}=\dfrac{1}{\sqrt{2\ln c}} \cdot (1+\overline{o}(1)),\quad c\to +\infty.
$$
Thus, taking Corollary~\ref{corollary3.2} into account, we finally obtain
\begin{gather*}
\overline{\mu}_t(c)=\dfrac{2}{c} \cdot \dfrac{e^{-\frac{L't}{8}}}{\sqrt{2\pi L't}} \cdot c^{3/2} \cdot \exp\left[-\frac{(\ln c)^2}{2L't}\right] \cdot \dfrac{\sqrt{L''}}{\sqrt{2\pi}} \cdot \sqrt{t} \cdot \dfrac{1}{\sqrt{2\ln c}} \cdot (1+\overline{o}(1))=\\
=\dfrac{e^{-\frac{L't}{8}}\sqrt{L''}}{\pi\sqrt{2L'}} \cdot \sqrt{\dfrac{c}{\ln c}} \cdot \exp\left[-\frac{(\ln c)^2}{2L't}\right] \cdot (1+\overline{o}(1)),\quad c\to +\infty.
\end{gather*}
The theorem is proved.
\end{proof}


\begin{thebibliography}{99}

\bibitem{Adler}
\textit{Adler~R.~J.}
On excursion sets, tube formulas and maxima of random fields~//
The Annals of Applied Probability.~-- 2000.~-- \textbf{10}, no.~1.~-- P.~1~--~74.

\bibitem{Arratia}
\textit{Arratia~R.~A.}
Coalescing Brownian motions on the line (PhD thesis).~--
University of Wisconsin, Madison, 1979.~-- iv~+~128~p.

\bibitem{AzaisWschebor}
\textit{Aza\"{\i}s~J.-M., Wschebor~M.}
Level sets and extrema of random processes and fields.~--
John Wiley \& Sons, Inc., 2009.~-- xi~+~393~p.

\bibitem{BorodinSalminen}
\textit{Borodin~A.~N., Salminen~P.}
Handbook of Brownian motion~-- facts and formulae.~--
2nd~ed.~-- Birkh\"{a}user, 2002.~-- xvi~+~658~p.

\bibitem{Dorogovtsev2004}
\textit{Dorogovtsev~A.~A.}
One Brownian stochastic flow~//
Theory of Stochastic Processes.~-- 2004.~-- \textbf{10}~(\textbf{26}), no.~3-4.~-- P.~21~--~25.

\bibitem{Dorogovtsev2007}
\textit{Dorogovtsev~A.~A.}
Measure-valued processes and stochastic flows.~--
Institute of Mathematics of the NAS of Ukraine, 2007.~-- 289~p. (in Russian)

\bibitem{Harris}
\textit{Harris~T.~E.}
Coalescing and noncoalescing stochastic flows in~$R_1$~//
Stochastic Processes and their Applications.~-- 1984.~-- \textbf{17}.~-- P.~187~--~210.

\bibitem{Kallenberg}
\textit{Kallenberg~O.}
Foundations of modern probability.~--
2nd~ed.~-- Springer, 2002.~-- xx~+~638~p.

\bibitem{Kunita}
\textit{Kunita~H.}
Stochastic flows and stochastic differential equations.~--
Cambridge University Press, 1990.~-- 346~p.

\bibitem{LeadbetterLindgrenRootzen}
\textit{Leadbetter~M.~R., Lindgren~G., Rootz\'en~H.}
Extremes and related properties of random sequences and processes.~--
Springer-Verlag, 1983.~-- xii~+~336~p.

\bibitem{LeadbetterSpaniolo}
\textit{Leadbetter~M.~R., Spaniolo~G.~V.}
Reflections on Rice's formulae for level crossings~-- history, extensions and use~//
Australian and New Zealand Journal of Statistics.~-- 2004.~-- \textbf{46}, no.~1.~-- P.~173~--~180.

\bibitem{Lifshits}
\textit{Lifshits~M.~A.}
Gaussian random functions.~--
TViMS, 1995.~-- 246~p. (in Russian)

\bibitem{LiptserShiryaev}
\textit{Liptser~R.~S., Shiryaev~A.~N.}
Statistics of random processes (non-linear filtering and related topics).~--
Nauka, 1974.~-- 696~p. (in Russian)

\bibitem{Nishchenko}
\textit{Nishchenko~I.~I.}
Discrete time approximation of coalescing stochastic flows on the real line~//
Theory of Stochastic Processes.~-- 2011.~-- \textbf{17}~(\textbf{33}), no.~1.~-- P.~70~--~78.

\bibitem{Zirbel}
\textit{Zirbel~C.~L.}
Stochastic flows: dispersion of a mass distribution and Lagrangian observations of a random field (PhD thesis).~--
Princeton University, 1993.~-- 162~p.

\end{thebibliography}
\end{document}